\documentclass[letterpaper, 12pt]{article}
\usepackage[margin=1in]{geometry}
\usepackage{amsmath,amsthm,amssymb}
\usepackage[longnamesfirst]{natbib}
\bibpunct[ ]{(}{)}{,}{a}{}{,}

\theoremstyle{plain}
\newtheorem{theorem}{Theorem}
\newtheorem{lemma}[theorem]{Lemma}
\newtheorem{proposition}[theorem]{Proposition}
\theoremstyle{definition}

\theoremstyle{remark}
\newtheorem{remark}{Remark}

\title{The Efron-Stein inequality for identically distributed pairs}
\author{Jnaneshwar Baslingker \and B\'alint Vir\'ag}

\begin{document}
\maketitle
\begin{abstract}
We prove that the classical Efron--Stein inequality holds for independent exchangeable pairs \((X_i,Y_i)\). The same inequality fails for independent identically distributed pairs; a simple trigonometric counterexample shows that the trivial Cauchy--Schwarz bound of factor \(n\) is sharp. When each random variable takes at most \(k_i\) values, a useful bound still holds with explicit constant \(\rho(k)\le\max_i k_i/2\).
\end{abstract}

\section{Introduction}

The classical Efron--Stein inequality \cite{efron1981jackknife,steele1986efron} bounds the variance of a function of independent random variables by a sum of squared differences obtained by replacing one variable at a time with an independent copy. It has become a fundamental tool in concentration of measure, influence theory, and the analysis of algorithms.

In this note we study a natural variant: instead of independent copies we are given, for each coordinate \(i\), an exchangeable pair \((X_i,Y_i)\) (i.e., \((X_i,Y_i)\stackrel{d}{=}(Y_i,X_i)\)) and the pairs are independent across \(i\). We show that the Efron--Stein inequality continues to hold in this setting (Theorem~\ref{thm: Efron-Stein for exchangeable pair}). The proof reduces the general exchangeable case to the classical Rademacher case by conditioning on the exchangeable \(\sigma\)-fields and a simple truncation argument.

When the pairs are merely identically distributed (but not exchangeable), the inequality can fail dramatically. We give a sharp counterexample (Proposition~\ref{p:id}) showing that the ratio of the two sides can be as large as the trivial Cauchy--Schwarz bound \(n\). On the positive side, when each variable takes values in a finite set of size at most \(k_i\), we obtain a useful bound whose constant \(\rho(k)\) depends only on the maximal cycle length appearing in the joint laws (Theorem~\ref{thm: finite support}). The proof combines Fourier analysis on cyclic groups with a Krein--Milman decomposition of circulations into cycle flows.

As an illustration we recover a sign-flipping analogue of the Gaussian Poincaré inequality. We also show that, without the identical-distribution assumption, even three-point supports do not improve upon the factor-\(n\) bound.

The exchangeable-pair technique and related concentration inequalities have been developed extensively by Chatterjee \cite{chatterjee2007stein} and are presented in detail in the monograph of Boucheron, Lugosi and Massart \cite{boucheron2013concentration}. The specific strengthening for exchangeable pairs and the finite-support analysis with cycle decomposition presented here appear to be new.

A standard version of the Efron-Stein inequality can be stated as follows.
\begin{theorem}[\cite{efron1981jackknife}, \cite{steele1986efron}]\label{thm: efron-stein}
Let \(X_1,\ldots, X_n,Y_1,\ldots,Y_n\) be \(2n\) independent random variables with \(X_i\stackrel{d}{=} Y_i\) for all \(i\). Let \(f\) be a measurable function, and let \(Z_i=f(X_1\ldots X_{i},Y_{i+1},\ldots Y_n)\) for \(i=0,\dots,n\). Then
$$
E(Z_n-Z_0)^2\le \sum_{i=1}^n E(Z_i-Z_{i-1})^2.
$$
\end{theorem}

Our goal is to prove the following strengthening:
\begin{theorem}\label{thm: Efron-Stein for exchangeable pair}
For each \(i\), let \((X_i,Y_i)\) be an exchangeable pair, so that the pairs \((X_i,Y_i)\) are independent as \(i\) varies. Let \(f\) be a measurable function, and let \(Z_i=f(X_1,\ldots X_{i},Y_{i+1},\ldots Y_n)\). Then
$$
E(Z_n-Z_0)^2\le \sum_{i=1}^n E(Z_i-Z_{i-1})^2.
$$
\end{theorem}

The proof is based on \cite{steele1986efron} with some significant modifications.
\begin{proof}
First we consider the special case when all \(X_i\) are uniform on \(\{-1,1\}\) and \(Y_i=-X_i\).

Since \(\{1,X_i\}\) is an orthonormal basis for the law of \(X_i\), \(\{\chi_s=\prod_{i\in s}X_i\,:\,s\subset [n]\}\) forms an orthonormal basis for the joint law of the \(X_i\). So we can write
$$f(X_1,\ldots, X_n)=\sum_{s\subset [n]
}c_s\chi_s$$ for some coefficients \(c_s\). We also have
$$f(Y_1,\ldots, Y_n)=\sum_{s\subset[n]}(-1)^{|s|}c_s\chi_s,$$ and the difference is given by
$$
Z_n-Z_0=2\sum_{s\subset [n],|s| \text{ odd}} c_s\chi_s.
$$
Similarly, we have
$$
Z_{i}-Z_{i-1}\stackrel{d}{=}f(X_1,\ldots,X_n)-f(X_1,\ldots,X_{i-1}, Y_{i},X_{i+1},\ldots X_n)=2\sum_{s\subset [n],i\in s}c_s\chi_s.
$$
Squaring, taking expectations and using orthonormality, the desired inequality reads
$$
2\sum_{s\subset [n],|s| \text{ odd}} c_s^2\le 2\sum_{s\subset [n]}c_s^2|s|
$$
which follows by summing \(2c_s^2\) times the inequality
$$
\mathbf 1_{|s|\text{ odd}}\le |s|
$$
over \(s\subset[n]\).

For general exchangeable pairs \((X_i,Y_i)\), and \(f\) bounded, given the exchangeable \(\sigma\)-field \(\mathcal I_i\) we know the set \(\{X_i,Y_i\}\) but not the order. So we can represent the conditional law of \((X_i,Y_i)\) as \(X_i=A_i+B_iX_i'\), \(Y_i=A_i+B_iY_i'\), where \(A_i=(X_i+Y_i)/2\), \(B_i=|X_i-Y_i|/2\) are \(\mathcal I_i\)-measurable, and \(X_i'=-Y_i'\) are uniform on \(\{-1,1\}\) independent of \(\mathcal I_i\). Conditioning on the \(\sigma\)-field \(\mathcal I\) generated by all the \(\mathcal I_i\), and setting \(f'(x_1,\ldots,x_n)=f(A_1+B_1x_1,\ldots ,A_n+B_nx_n)\), and applying the first part to the primed quantities,
we get
$$
E[(Z_n-Z_0)^2|\mathcal I]\le \sum_{i=1}^n E[(Z_i-Z_{i-1})^2|\mathcal I].
$$
Taking expectations completes the proof.

To handle the unbounded case, note that it holds for \(f_b=(-b \vee f)\wedge b\). Since \(x\mapsto (-b \vee x)\wedge b\) is a contraction with respect to Euclidean distance on \(\mathbb R\), both sides of the inequality are monotone increasing in \(b\). The monotone convergence theorem then implies the general case.
\end{proof}

Let \((Z_1,Z_2,...,Z_n)\) be i.i.d \(N(0,1)\) random variables and \((B_1,B_2,...,B_n)\) are i.i.d (independent of \(Z_i\)) with \(\mathbb{P}(B_1=1)=p\) and \(\mathbb{P}(B_1=-1)=1-p\). Taking \(Y_i=B_iZ_i\) and applying Theorem \ref{thm: Efron-Stein for exchangeable pair} gives,
    \[
    E[(Z_n-Z_0)^2]\leq \sum_{i=1}^n E(Z_i-Z_{i-1})^2
    \]
where \begin{align*}
    E(Z_{n}-Z_{n-1})^2&=(1-p)E[(f(Z_1,Z_2,...,Z_n)-f(Z_1,Z_2,...,-Z_n))^2]\\
    &\leq 4(1-p)E\left[\left(\frac{\partial f}{\partial z_n}\right)^2\right].
\end{align*}
As a result we obtain
\[
E[(Z_n-Z_0)^2]\le 4(1-p)\sum_{i=1}^nE\left[\left(\frac{\partial f}{\partial z_i}\right)^2\right],
\]
which is the sign-flipping analogue of the classical Gaussian Poincaré inequality.

Next, we show that the inequality fails for identically distributed pairs.
\begin{proposition}\label{p:id}
Let \(\epsilon\in \mathbb R\), let \(X_i\) be uniform on \([0,2\pi]\) and let \(Y_i\) equal \(X_i+2\epsilon\) mod \(2\pi\). Then \(X_i\stackrel{d}{=}Y_i\). Assume that the pairs \((X_i,Y_i)\) are independent as \(i\) varies. Let \(f(x_1,\ldots x_n)=\sin(x_1+\ldots +x_n)\), and let \(Z_i=f(X_1,\ldots X_{i},Y_{i+1},\ldots Y_n)\). Then the random variables \(Z_i\) (and hence the ratio on the left-hand side of \eqref{e}) depend on \(\epsilon\), and we have
\begin{equation}\label{e}
\lim_{\epsilon\to 0} \frac
{E(Z_n-Z_0)^2}{\sum_{i=1}^n E(Z_i-Z_{i-1})^2}=n.
\end{equation}
\end{proposition}

This is as bad as possible, since for arbitrary random variables \(Z_i\) by Cauchy-Schwarz applied to the sequences \(Z_{i+1}-Z_i\) and \(1\),
\begin{equation}\label{e:CS}
E(Z_n-Z_0)^2\le n\sum_{i=1}^nE(Z_i-Z_{i-1})^2.
\end{equation}
\begin{proof}
With \(S=X_1+\ldots +X_n\) we have
$$
Z_0-Z_n=\sin(S+2n\epsilon)-\sin(S)=2\cos(S+n\epsilon)\sin(n\epsilon),
$$
$$
E(Z_0-Z_n)^2=2\sin^2(n\epsilon).
$$
Similarly,
$$
E(Z_i-Z_{i-1})^2=2\sin^2(\epsilon).
$$
So \eqref{e} equals to
\[
\lim_{\epsilon\to 0}\frac{\sin^2(n\epsilon)}{n\sin^2(\epsilon)}=n.
\qedhere\]
\end{proof}

\section*{Finite support case}

Inspired by the counterexample of Proposition \ref{p:id}, we may ask if a useful bound still exists when \(X_i\) can take only \(k_i\) values. For \(k\in\mathbb{Z}^n,\) the constant factor we get is given by
$$
\rho(k)=\max_{u,m}\frac{\sin^2(\sum_{i} \pi u_i/m_i)}{\sum_{i}\sin^2(\pi u_i/m_i)},
$$
where the maximum is over all \(m, u\in \mathbb Z^n\), satisfying \(1\le m_i\le k_i\), and \(0\le u_i<m_i\) for all \(i\).

We can bound this quantity as follows.
\begin{lemma}
    \(\rho(k)\le \max_i k_i/{2}\).
\end{lemma}
\begin{proof}
Let \(\ell=|\{i:u_i\neq 0\}|\) and \(\kappa=\max\limits_{i}k_i.\) Using \(\left\lvert\sin\left(\sum\limits_{i=1}^{\ell}x_i\right)\right\rvert\leq \sum\limits_{i=1}^\ell\left\lvert\sin(x_i)\right\rvert\) and Cauchy-Schwarz, we get \( \rho(k)\leq \ell\). For \(u_i\neq 0\),
\[
\sin\left(\frac{\pi u_i}{m_i}\right)\geq\sin\left(\frac{\pi }{m_i}\right)\geq \sin\left(\frac{\pi}{\kappa}\right)\geq \frac{2}{\kappa}.
\] 
This gives \(\rho(k)\leq \frac{\kappa^2}{4\ell}\). Hence
\[
\rho(k)\leq \ell\wedge \left(\frac{\kappa^2}{4\ell}\right)\leq \frac{\kappa}{2}.
\qedhere
\]
\end{proof}

\begin{theorem}\label{thm: finite support}
For each \(i\) let \((X_i,Y_i)\) be an identically distributed pair, so that the pairs \((X_i,Y_i)\) are independent as \(i\) varies, and \(X_i\) can take at most \(k_i\) values. Let \(f\) be a bounded measurable function, and let \(Z_i=f(X_1,\ldots X_{i},Y_{i+1},\ldots Y_n)\). Then
$$
E(Z_n-Z_0)^2\le \rho(k) \sum_{i=1}^n E(Z_i-Z_{i-1})^2, \qquad \rho (k)\le \max_i k_i/2.
$$
\end{theorem}
\begin{proof}
It helps to allow complex-valued \(f\), in which case the result applies to the squared modulus:
$$
E|Z_n-Z_0|^2\le \frac{\kappa}{2}\sum_{j=1}^n E|Z_j-Z_{j-1}|^2
$$
First we consider the special case when all \(X_j\) is uniform on \(\mathbb Z_{k_j}=\{0,\ldots,k_j-1\}\), and \(Y_j=X_{j}+1\) mod \(k_j\). Then \(\{e^{2\pi i v X_j/k_j}, v\in \mathbb Z_{k_j}\}\) is an orthonormal basis for the law of \(X_j\), and so
\(\{\chi_u=\prod_j e^{2\pi iu_j X_j/k_j}\ :\,u\in \prod_j \mathbb Z_{k_j}=\mathcal U\}\) forms an orthonormal basis for the joint law of the \(X_j\) over the complex numbers. So we can write
$$f(X_1,\ldots, X_n)=\sum_{u\in \mathcal U}c_u\chi_u$$ for some coefficients \(c_u\). We also have
$$f(Y_1,\ldots, Y_n)=\sum_{u\in \mathcal U}\eta_uc_u\chi_u,$$
where \(\eta(u)=\prod_j \eta_j^{u_j}\), \(\eta_j=\exp(2\pi/k_j)\). We have
$$
Z_n-Z_0=\sum_{u\in \mathcal U} (1-\eta_u)c_u\chi_u.
$$
Similarly, we have
$$
Z_{j}-Z_{j-1}\stackrel{d}{=}f(X_1,\ldots,X_n)-f(X_1,\ldots,X_{j-1}, Y_{j},X_{j+1},\ldots X_n)=\sum_{u\in \mathcal U}(1-\eta_j^{u_j})c_u\chi_u.
$$
Squaring, taking expectations and using orthonormality, the desired inequality reads
$$
\sum _{u\in \mathcal U}|c_u|^2|1-\prod_j\eta_j^{u_j}|^2\le \rho(k)\sum _{u\in \mathcal U}|c_u|^2\sum_j|1-\eta_j^{u_j}|^2,
$$
which follows by summing \(|c_u|^2\) times the inequality
$$
|1-\prod_j\eta_j^{u_j}|^2\le \rho(k)\sum_j|1-\eta_j^{u_j}|^2,
$$
which translates to
$$
\sin^2(\sum_{j} \pi u_j/k_j)\le \rho(k)\sum_{j}\sin^2(\pi u_j/k_j)
$$
holding by definition of \(\rho(k)\).

For general identically distributed pairs \((X_j,Y_j)\) on \([k_j]\), we note that the joint mass function \(p_j(x,y)\) defines a sourceless flow on \([k_j]\). By the Krein-Milman theorem, the flow can be decomposed as a convex combination of unit directed flows over cycles. Since the desired inequality is linear in each of the joint laws \((X_j,Y_j)\) (but not in \(f\)), it suffices to show it for each component. Denoting the cycle lengths by \(k_j'\), and relabeling, this is exactly what was shown in the first part of the proof.
\end{proof}

\begin{remark}
It follows from the proof that the bound \(\rho(k)\) can be replaced by
\[
\max_i\bigl(\text{maximal cycle length in the support of }(X_i,Y_i)\bigr)/2.
\]
In particular, when each pair \((X_i,Y_i)\) is exchangeable the maximal cycle length is at most 2 and one recovers the classical Efron--Stein constant 1 (Theorem~\ref{thm: Efron-Stein for exchangeable pair}). When the pairs are independent the same conclusion holds.
\end{remark}

\section*{Different law for \(X_i,Y_i\)}

If \(X_i\) is allowed to have different law from \(Y_i\), we cannot improve on the Cauchy-Schwarz bound \eqref{e:CS} even if we require only three possible values for \(X_i,Y_i\) and \(EZ_i=0\). For this, we set \(Y_i=0\), \(X_i\) to be independent and uniform on \(\{-1,1\}\), and 
$$
f(x_1,\ldots,x_n)=x_1(|x_1|+\ldots+|x_n|).
$$
Then \(Z_i=iX_1\), \(EZ_i=0\), and
$$
(Z_0-Z_n)^2=n^2, \qquad \sum_{i=1}^n
(Z_{i}-Z_{i-1})^2=n.
$$

\begin{remark}
There are several generalizations and versions of the Efron-Stein inequality, including the OSSS inequality \cite{o2005every}.
\end{remark}

\noindent {\bf Acknowledgments.} We thank G\'abor Pete for some useful comments on a previous version.

\bibliographystyle{dcu}
\bibliography{shape}
\end{document}